\documentclass[a4paper,11pt]{article}

\usepackage{fullpage,enumerate}
\usepackage{amsmath,amsthm}
\usepackage{tikz}

\newcommand*{\Wm}[1][k]{\mathsf{Wm}(#1)}
\newcommand{\ex}{\operatorname{ex}}

\newtheorem{thm}{Theorem}
\newtheorem{lem}[thm]{Lemma}
\newtheorem{cor}[thm]{Corollary}
\newtheorem{conj}[thm]{Conjecture}
\newtheorem{prop}[thm]{Proposition}
\theoremstyle{remark}
\newtheorem*{rem}{Remark}

\usepackage{authblk}

\title{Lower bounds for the Randi\'c index in terms of matching number}
\author[1]{Saieed Akbari\footnote{s\_akbari@sharif.edu}}
\author[1]{Sina Ghasemi Nezhad\footnote{sina.ghaseminejad@gmail.com}}
\author[1]{Reyhaneh Ghazizadeh\footnote{reyyghazizade@gmail.com}}
\author[2]{John Haslegrave\footnote{j.haslegrave@lancaster.ac.uk}}
\author[1]{Elahe Tohidi\footnote{elahetohidi2021@gmail.com}}
\affil[1]{Department of Mathematical Sciences, Sharif University of Technology, Tehran, Iran}
\affil[2]{School of Mathematical Sciences, Lancaster University, Lancaster, UK}

\date{}

\begin{document}
	\maketitle
\begin{abstract}
We investigate how small the Randi\'c index of a graph can be in terms of its matching number, and prove several results. We give best-possible linear bounds for graphs of small excess and for subcubic graphs; in the former case the size of excess we permit is qualitatively the best possible. We show that a linear bound holds for any sparse hereditary graph class (such as planar graphs). In general, however, we show that it can be much smaller than linear. We determine the asymptotic growth rate of the minimum Randi\'c index for graphs with a nearly-perfect matching, and conjecture that the same bounds hold for all graphs.
\end{abstract}
\section{Introduction}
In 1975, the chemist Milan Randi\'c \cite{randic1975characterization} proposed a topological index under the name \textit{branching index}, suitable for measuring the extent of branching of the carbon-atom skeleton of saturated hydrocarbons \cite{li2008survey}. While this was some time later than the Wiener index, proposed with a similar purpose in 1947 \cite{wiener}, it quickly became one of the most studied topological indices, and, together with its generalizations, has been the subject of multiple books; see e.g.\ \cite{li-gutman}.

For a simple graph $G$, the \textit{Randi\'c index} is defined as 
\[R(G) = \displaystyle\sum_{uv\in E(G)} \frac{1}{\sqrt{d(u)d(v)}};\]
we shall refer to the quantity $1/\sqrt{d(u)d(v)}$ as the \textit{weight} of the edge $uv$. While the main focus of study on the Randi\'c index is for connected graphs, owing to its origins as a molecular descriptor, often the same bounds hold for all graphs without isolated vertices. Trees, corresponding as they do to open-chain compounds, bounded degree graphs, and sparse graphs more generally, all arise naturally in this context.

Many extremal results are known for the Randi\'c index. In particular it is at most $|G|/2$ (where, as usual, we write $|G|$ for the number of vertices of $G$), which is achieved by any non-empty regular graph, or a disjoint union of such graphs. This was known at least as early as 1986 \cite{faj86}. The maximum value for trees is achieved by the path \cite{path}, and is $|G|/2+\sqrt{2}-1.5$, very close to the overall maximum. The minimum value among trees, and more generally graphs with no isolated vertices, is $\sqrt{|G|-1}$, achieved by the star \cite{bollobas1998graphs}. There has been significant interest in obtaining improved lower bounds for the Randi\'c index in terms of other graph parameters, such as radius \cite{faj86}, diameter \cite{DLS11} or domination number \cite{BNR20}, and similar questions have been studied for other related topological indices \cite{BF16,ZZ24}. Often stronger bounds apply to trees than to general graphs.

In the spirit of this line of research, and motivated by the observation that trees with small Randi\'c index also have small matching number, we investigate the relationship between matching number and lower bounds for the Randi\'c index. (As shown by the star, no upper bound on the Randi\'c index purely in terms of the matching number exists.) 

Suppose a graph $G$ has matching number $\alpha'(G)=k$. How small can the Randi\'c index $R(G)$ be? If $G$ is a tree we always have 
\begin{equation}R(G)>\frac{\alpha'(G)}{\sqrt{2}},\label{tree-bound}\end{equation} 
which follows from a more precise result in \cite{trees}, and \eqref{tree-bound} also applies to unicyclic graphs \cite{unicyclic}. In fact a corresponding lower bound of $2^a k$ applies to the general Randi\'c index $R_a(G)$ with exponent $a\in[-1/2,0]$ (the normal Randi\'c index being the case $a=-1/2$); see \cite{general}.  However, \eqref{tree-bound} is not true in general, and indeed we shall show in Section \ref{sec:windmill} that no linear bound holds for all graphs.

In Section \ref{sec:excess} we consider how far from a tree a graph can be and still be guaranteed to satisfy \eqref{tree-bound}. We measure distance from a tree by the \textit{excess} $\ex(G):=|E(G)|-|G|+1$ (also known as the cyclomatic number). We show that \eqref{tree-bound} applies if $\ex(G)$ is sufficiently small compared to $\sqrt{|G|}$, and this is qualitatively the best possible. 

In Section \ref{sec:cubic} we turn to \textit{subcubic graphs}, that is, graphs of maximum degree at most $3$. Although these graphs are sparse they can have linear excess and do not necessarily satisfy \eqref{tree-bound}. However, we prove a corresponding linear bound with the best possible constant.

In Section \ref{sec:sparse} we consider sparse hereditary classes more generally, that is, graphs all of whose induced subgraphs have bounded average degree. Our initial motivation here was the class of planar graphs. We show that any such class obeys a linear lower bound (but that other sparse classes do not).

Finally, in Section \ref{sec:windmill} we consider what bounds hold for arbitrary graphs with matching number $k$. We show that $R(G)$ can be $O(k^{2/3})$, and prove that this is the best possible for graphs with a perfect or nearly-perfect matching (that is, a matching with at most one unmatched vertex). We conjecture that the same lower bound applies in general.
\section{Graphs with small excess}\label{sec:excess}
In this section we prove the following.
\begin{thm}\label{quad-excess}Let $G$ be a connected graph. If $\ex(G)\leq\sqrt{|G|}/28$, then $R(G)> \alpha'(G)/\sqrt{2}$.
\end{thm}

Theorem \ref{quad-excess} is the best possible in two senses. First, the constant $1/\sqrt{2}$ cannot be improved even for excess $0$ (trees); see \cite{trees}. Secondly, the condition $\ex(G)\leq \sqrt{|G|}/28$ cannot be qualitatively improved, in that there exist graphs with $\ex(G)=O(\sqrt n)$ for which the conclusion fails.

Define the \textit{broken windmill} $\mathsf{BW}(a,b)$ to be the graph consisting of $a+b$ independent edges with one extra vertex adjacent to both ends of $a$ edges and one end of the other $b$ edges. Clearly we have $\alpha'(\mathsf{BW}(a,b))=a+b$ and $|\mathsf{BW}(a,b)|=2a+2b+1=:n$ and $\ex(\mathsf{BW}(a,b))=a$. We can calculate
\[R(\mathsf{BW}(a,b))=\frac{a}{2}+\frac{b}{\sqrt{2}}+\frac{1}{\sqrt{2}}\sqrt{2a+b}.\]
In particular, this is less than $\frac{a+b}{\sqrt{2}}$ whenever
\[\frac{2a+b}{2}<a^2\Big(\frac{1}{\sqrt 2}-\frac{1}{2}\Big)^2,\]
that is whenever $n+2a-1<a^2(3-\sqrt{8})$, or equivalently $a^2-2(3+\sqrt 8)a-(3+\sqrt 8)(n-1)>0$. This holds when
\[a>3+\sqrt{8}+\sqrt{(3+\sqrt{8})n+14+5\sqrt{8}}.\]
Thus, taking the smallest such $a$ for each odd $n\geq 49$, we obtain a suitable sequence of graphs with $\ex(G)\sim(1+\sqrt{2})\sqrt{|G|}$. 

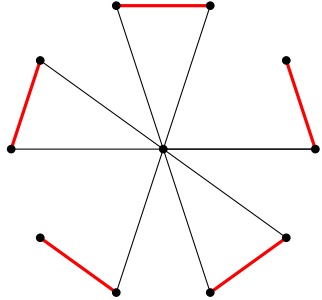
\begin{figure}[ht]
	\centering
	\begin{tikzpicture}
		\foreach\x in {0,72,108,144,180,252,288,324,360}{
			\draw (0:0) -- (\x:2);
		}
		\foreach\x in {0,72,...,288}{
			\draw[very thick, red] (\x:2) -- (\x+36:2);
		}
		\foreach\x in {0,36,...,324}{
			\filldraw (\x:2) circle (.05);
		}
		\filldraw (0:0) circle (.05);
	\end{tikzpicture}	
	\caption{The broken windmill $\mathsf{BW}(3,2)$, with a maximum matching highlighted.}
\end{figure}
\begin{rem}
	We observe that \eqref{tree-bound} also holds for graphs with sufficiently large excess. Indeed, Bollob\'as and Erd\H{o}s  \cite{bollobas1998graphs} showed that any graph $G$ with $m$ edges has $R(G)\geq \frac{2m}{\sqrt{8m+1}-1}$, which is attained for complete graphs. It follows that any graph $G$ with $n$ vertices and at least $\frac 12\binom n2$ edges satisfies $R(G)>\frac{n}{2\sqrt 2}\geq\frac{\alpha'(G)}{\sqrt 2}$, and similarly graphs with any given fraction of the possible edges satisfy a linear lower bound. This is best possible, since generalized windmills (see the proof of Proposition \ref{prop:windmill}) with appropriate parameters show that no number of edges that is $o(n^2)$ but $\omega(n)$ would imply a linear bound.\end{rem}

\begin{lem}\label{excess-bound}Let $H$ be a connected graph with $n$ vertices and $n+k$ edges, such that every leaf is adjacent to a different vertex of degree at least $4$.
	Then $R(H)\geq \frac{n-7k}{2}$.
\end{lem}
Note that we may subdivide non-leaf edges of $H$ arbitrarily without violating the condition. Since subdividing edges does not change the excess, $n$ can be arbitrarily large for any given $k$. 
\begin{proof}First consider the graph $H'$ obtained by removing all leaves. Note that this preserves the excess, i.e.\ $H'$ has $n'$ vertices and $n'+k$ edges for some $n'\leq n$. Additionally $H'$ has minimum degree at least $2$. Let $t$ be the number of vertices of degree exceeding $2$ in $H'$. Note that the sum of degrees of these $t$ vertices is exactly $2t+2k$, since the other vertices have total degree $2n'-2t$. Since it is also at least $3t$, we have $t\leq 2k$. Thus at most $6k$ edges meet vertices of degree at least $3$ in $H'$. Consequently at least $n'-5k$ edges are between vertices of degree $2$ in $H'$. The conditions on $H$ imply that these vertices also have degree $2$ in $H$, and that $n-n'\leq t\leq 2k$. Thus at least $n-7k$ edges of $H$ have weight $1/2$, which immediately gives the required bound.
\end{proof}
\begin{rem}This bound is certainly not optimal, but is the best we can do with such a simple argument, since our bounds are tight individually although not in combination.
\end{rem}

We now proceed to the proof of Theorem \ref{quad-excess}. Recall that the \textit{$k$-core} of a graph is the largest subgraph of minimum degree $k$ (if one exists), and can be obtained by iteratively removing vertices of degree less than $k$. In particular, any graph that is not a forest has a $2$-core.
\begin{proof}[Proof of Theorem \ref{quad-excess}]
	Our approach is to reduce to a graph $H$ which is close to the $2$-core of $G$, and then apply Lemma \ref{excess-bound}. We construct a sequence of graphs $G=G_0,G_1,\ldots,G_r$ according to the following rules.
	\begin{enumerate}[(i)]
		\item If $G_i$ has a leaf $u$ such that $\alpha'(G_i\setminus\{u\})=\alpha'(G_i)$, choose one such $u$ and set $G_{i+1}=G_i\setminus\{u\}$.
		\item If (i) does not apply, but $G_i$ has a leaf $u$ with neighbor $v$ of degree $2$, choose one such pair $u,v$ and set $G_{i+1}=G_i\setminus\{u,v\}$.
		\item If neither (i) nor (ii) applies, end the construction and set $r=i$.
	\end{enumerate}
	Since we may have a choice of leaf to remove, neither the sequence nor the final graph $G_r$ is uniquely determined. Note that since we remove the same number of edges as vertices we always have $\ex(G_{i+1})=\ex(G_i)$.
	
	In Case (i), let $v$ be the neighbor of $u$. Then we have $\alpha'(G_{i+1})=\alpha'(G_i)$ and 
	\begin{align*}
		R(G_i) - R(G_{i+1})&= \frac{1}{ \sqrt {d_{G_i}(v)} }+\Big(\frac{1}{\sqrt{d_{G_i}(v)}}-\frac{1}{\sqrt{d_{G_i}(v)-1}}\Big)\Big(\displaystyle\sum _{\begin{subarray}{c} 
				w \sim v \\ 
				w \neq u 
		\end{subarray} }{\frac{1}{\sqrt{d_{G_i}(w)}} }\Big)\\
		&\geq\frac{1}{ \sqrt {d_{G_i}(v)} }+\Big(\frac{1}{\sqrt{d_{G_i}(v)}}-\frac{1}{\sqrt{d_{G_i}(v)-1}}\Big)(d_{G_i}(v)-1)\\
		&=\frac{1}{\sqrt{d_{G_i}(v)}+\sqrt{d_{G_i}(v)-1}}\\
		&> \frac{1}{2\sqrt{n}}.
	\end{align*}
	
	In Case (ii), since (i) does not apply, every edge adjacent to a leaf must be part of every maximum matching. In particular, no two leaves have the same neighbor. Let $w\neq u$ be the other neighbor of $v$. Now $\alpha'(G_{i+1})=\alpha'(G_i)-1$ and
	\begin{align*}
		R(G_i) - R(G_{i+1})&=\frac{1}{\sqrt{2}}+\frac{1}{ \sqrt {2d_{G_i}(w)} }+\Big(\frac{1}{\sqrt{d_{G_i}(w)}}-\frac{1}{\sqrt{d_{G_i}(w)-1}}\Big)\Big(\displaystyle\sum _{\begin{subarray}{c} 
				x \sim w \\ 
				x \neq v 
		\end{subarray} }{\frac{1}{\sqrt{d_{G_i}(x)}} }\Big)\\
		&\geq \frac{1}{\sqrt{2}}+\frac{1}{ \sqrt {2d_{G_i}(w)} }+\Big(\frac{1}{\sqrt{d_{G_i}(w)}}-\frac{1}{\sqrt{d_{G_i}(w)-1}}\Big)\Big(\frac{d_{G_i}(w)-2}{\sqrt{2}}+1\Big)\\
		&=\frac{1}{\sqrt{2}}+\Bigg[\frac{1}{\sqrt{2}}-\Big(1-\frac{1}{\sqrt{2}}\Big)\frac{1}{\sqrt{d_{G_i}(w)(d_{G_i}(w)-1)}}\Bigg]\frac{1}{\sqrt{d_{G_i}(w)}+\sqrt{d_{G_i}(w)-1}}.
	\end{align*}
	Note that, since $d_{G_i}(w)\geq 2$, we have
	\[\frac{1}{\sqrt{2}}-\Big(1-\frac{1}{\sqrt{2}}\Big)\frac{1}{\sqrt{d_{G_i}(w)(d_{G_i}(w)-1)}}\geq\frac{1}{\sqrt{2}}-\Big(1-\frac{1}{\sqrt{2}}\Big)\frac{1}{\sqrt{2}}=\frac{1}{2}.\]
	Thus, 
	\[R(G_i) - R(G_{i+1})\geq\frac{1}{\sqrt{2}}+\frac{1}{2}\cdot\frac{1}{\sqrt{d_{G_i}(w)}+\sqrt{d_{G_i}(w)-1}}>\frac{1}{\sqrt{2}}+ \frac{1}{4\sqrt{n}}.
	\]
	
	In either case we have $R(G_i)\geq R(G_{i+1})+\frac{\alpha'(G_i)-\alpha'(G_{i+1})}{\sqrt{2}}+\frac{|G_i|-|G_{i+1}|}{8\sqrt{n}}$. Thus we also have 
	\begin{equation}\label{g-to-gr}R(G)\geq \frac{\alpha'(G)}{\sqrt{2}}+R(G_r)-\frac{\alpha'(G_r)}{\sqrt{2}}+\frac{n-|G_r|}{8\sqrt{n}}.\end{equation}
	
	If $\ex(G)=0$, i.e.\ $G$ is a tree, it is easy to see that either (i) or (ii) must apply at each stage where $G_i$ has more than $2$ vertices. Thus $|G_r|\leq 2$ and \eqref{g-to-gr} gives $R(G)> \frac{\alpha'(G)}{\sqrt{2}}$.
	
	Otherwise, we obtain a non-trivial $G_r$ which includes the entire $2$-core of $G$ and at most one leaf adjacent to each vertex of the $2$-core. Consider the graph $H$ obtained by removing all leaves adjacent to vertices of degree $3$ in $G_r$. Then $|H|\geq |G_r|/2$, and $H$ satisfies the conditions of Lemma \ref{excess-bound}. Thus we have $R(H)\geq\frac{|H|-7\ex(H)+7}{2}=\frac{|H|-7\ex(G)+7}{2}$. 
	
	Now we compare $R(G_r)$ and $R(H)$. For each leaf of $G_r$, adding it to $H$ creates a new edge of weight $1/\sqrt 3$ but reduces the weights of two existing edges by a multiplicative factor of $\sqrt{2/3}$. Since these edges previously had weight at most $1/2$, adding the leaf increases the total weight by at least $1/\sqrt 3-(1-\sqrt{2/3})$. Thus we have
	\[R(G_r)\geq \frac{|H|-7\ex(G)+7}{2}+(|G_r|-|H|)\Big(\frac{1}{\sqrt 3}+\frac{\sqrt{2}}{\sqrt{3}}-1\Big).\]
	This is increasing in $|H|$, so minimised (for fixed values of $|G_r|,\ex(G)$) when $|H|$ is as small as possible, i.e.\ $|H|=|G_r|/2$. Thus we obtain
	\[R(G_r)\geq \Big(\frac{1}{\sqrt{6}}+\frac{1}{\sqrt{12}}-\frac{1}{4}\Big)|G_r|-\frac{7\ex(G)-7}{2}.\]
	Finally, note that
	\[\Big(\frac{1}{\sqrt{6}}+\frac{1}{\sqrt{12}}-\frac{1}{4}\Big)|G_r|>\Big(\frac{1}{2\sqrt{2}}+\frac{1}{11}\Big)|G_r|\geq\frac{\alpha'(G_r)}{\sqrt{2}}+\frac{|G_r|}{11}.\]
	
	Combining these bounds with \eqref{g-to-gr} gives
	\[R(G)>\frac{\alpha'(G)}{\sqrt{2}}+\frac{|G_r|}{11}+\frac{n-|G_r|}{8\sqrt{n}}-\frac{7\ex(G)-7}{2}.\]
	Since $8\sqrt{n}\geq 8\sqrt 3>11$, this is increasing in $|G_r|$, giving
	\[R(G)>\frac{\alpha'(G)}{\sqrt{2}}+\frac{\sqrt{n}}{8}-\frac{7\ex(G)-7}{2},\]
	which gives the required bound since $28(\ex(G)-1)\leq \sqrt{n}$.
\end{proof}

\section{Subcubic graphs}\label{sec:cubic}
The \textit{corona product} $G \circ H$ of two graphs $G$ and $H$ is obtained by taking one copy of $G$ and $|V(G)|$ copies of $H$, and by joining each vertex of the $i$-th copy of $H$ to the $i$-th vertex of $G$, where $1 \leq i \leq |G|$. In particular, the corona product $G \circ K_1$ is formed by adding a pendant edge to every vertex of $G$.

\begin{thm} \label{thm-cubic}
	Let $G$ be a subcubic graph. Then $R(G) \geq \Big( \frac{1}{\sqrt{3}} + \frac{1}{3} \Big) \alpha^\prime (G)$ and equality holds if and only if $G$ is (after removing any isolated vertices) of the form $H \circ K_1$ for some $2$-regular graph $H$.
\end{thm}
It is natural to wonder whether Theorem \ref{thm-cubic} applies more generally, in that for any $r\geq 3$ the best-possible constant for graphs of maximum degree at most $r$ comes from graphs of the form $H\circ K_1$ with $H$ being $(r-1)$-regular. (Note that some linear bound does hold by Corollary \ref{planar}.) However, this cannot be true, since it would give a constant of $\frac{1}{\sqrt r}+\frac{r-1}{2r}>1/2$, implying $R(G)>\alpha'(G)/2$ for all graphs and consequently contradicting Proposition \ref{prop:windmill}. 

Our proof of Theorem \ref{thm-cubic} uses the following interesting result proved by O and Shi \cite{suil2018sharp} (see \cite{Has24} for a shorter proof).

\begin{thm}\label{thm-delta}
	Let $G$ be a graph with maximum degree $\Delta$ and minimum degree $\delta$. Then $R(G) \geq \frac{\sqrt{\delta \Delta}}{\delta + \Delta}|G|$.
\end{thm}

\begin{proof}[Proof of Theorem \ref{thm-cubic}]
	If $\Delta(G)\leq 2$, then every component is a path or cycle and it is not hard to see that the assertion holds. Thus assume that $\Delta(G)=3$. We may assume $\delta(G)>0$ since removing isolated vertices does not change $R(G)$ or $\alpha'(G)$. First suppose that $\delta(G)\geq 2$. By Theorem \ref{thm-delta}, we have $R(G) \geq \frac{\sqrt{6}}{5}|G| > \Big(\frac{1}{3}+\frac{1}{\sqrt{3}}\Big) \frac{|G|}{2} \geq \Big(\frac{1}{3}+\frac{1}{\sqrt{3}}\Big) \alpha^\prime (G)$.
	
	Thus assume that $u$ is a pendant vertex and $uv \in E(G)$. First suppose that $d(v)=2$ and $N_G(v)=\{u, w\}$. If $d(w)=2$, and $N_G(w)=\{v, z\}$, then define $G^\prime=G \backslash\{u, v\}$. We have,
	\begin{equation*}
		R(G)=R(G^\prime) + \frac{1}{\sqrt{2}} + \frac{1}{2} + \frac{1}{\sqrt{2 d(z)}} -\frac{1}{\sqrt{d(z)}} .
	\end{equation*}
	Now, by the induction hypothesis we have
	\begin{equation*}
		R(G) \geq \Big(\frac{1}{3}+\frac{1}{\sqrt{3}}\Big) (\alpha^\prime (G) - 1) + \frac{1}{\sqrt{2}} + \frac{1}{2} + \frac{1}{\sqrt{2 d(z)}} - \frac{1}{\sqrt{d(z)}}.
	\end{equation*}
	Since $d(z) \in\{2,3\}$, we obtain $R(G) > \Big(\frac{1}{3}+\frac{1}{\sqrt{3}}\Big) \alpha^{\prime}(G)$.
	Now, assume that $d(w)=3$. Let $N_G(w)=\{v, t, z\}$. In this case, we have
	\begin{equation*}
		R(G)=R(G^{\prime}) + \frac{1}{\sqrt{2}} + \frac{1}{\sqrt{6}} + \frac{1}{\sqrt{3 d(t)}} + \frac{1}{\sqrt{3 d(z)}} - \frac{1}{\sqrt{2 d(t)}} - \frac{1}{\sqrt{2 d(z)}}.
	\end{equation*}
	Thus we find that,
	\begin{align*}
		R(G) & \geq  \Big(\frac{1}{3}+\frac{1}{\sqrt{3}}\Big) (\alpha^\prime (G) - 1) +  \Big(\frac{1}{\sqrt{d(t)}}+\frac{1}{\sqrt{d(z)}}\Big) \Big(\frac{1}{\sqrt{3}}-\frac{1}{\sqrt{2}}\Big) + \frac{1}{\sqrt{2}} + \frac{1}{\sqrt{6}} \\
		& \geq \Big(\frac{1}{3}+\frac{1}{\sqrt{3}}\Big) \alpha^\prime (G) - \frac{1}{3} - \frac{1}{\sqrt{3}}+\frac{2}{\sqrt{3}}\Big(\frac{1}{\sqrt{3}}-\frac{1}{\sqrt{2}}\Big)+\frac{1}{\sqrt{2}}+\frac{1}{\sqrt{6}} \\
		& > \Big(\frac{1}{3}+\frac{1}{\sqrt{3}}\Big) \alpha^{\prime}(G).
	\end{align*}
	Thus one may assume that every pendant vertex is adjacent to a vertex of degree $3$. Now, let $d(u)=d(v)=2$, and $uv \in E(G)$. Suppose that $N_G(u)=\{t, v\}$ and $N_G(v)=\{u, w\}$.
	
	First we deal with the case that $t=w$. Here we consider three subcases: $d(w)=2$; $d(w)=3$ and the other neighbor of $w$ has degree $1$; $d(w)=3$ and the other neighbor of $w$ has degree at least $2$. In the first two subcases, the component containing $uv$ has $3$ or $4$ vertices respectively. Let this component be $G_1$ and the remaining graph (if any) be $G_2$. By the induction hypothesis, we have 
	\[R(G_2)\geq\Big(\frac{1}{3}+\frac{1}{\sqrt{3}}\Big)\alpha^\prime (G_2),\] 
	where we take $R(G_2)=\alpha^\prime(G_2)=0$ if $G_2$ has no vertices. In the first subcase, we have $\alpha^\prime(G)=\alpha^\prime(G_2)+1$ and 
	\[R(G)=R(G_2)+\frac 32>R(G_2)+\frac{1}{3}+\frac{1}{\sqrt{3}};\]
	in the second we have $\alpha^\prime(G)=\alpha^\prime(G_2)+2$ and
	\[R(G)=R(G_2)+\frac 12+\frac2{\sqrt 6}+\frac13>R(G_2)+2\Big(\frac{1}{3}+\frac{1}{\sqrt{3}}\Big);\]
	in either case the result follows.
	
	In the final subcase, write $G'=G\setminus\{u,v\}$, and let $y$ be the neighbor of $w$ in $G'$. Since $d_{G'}(y)=d_G(y)\geq 2$, we have $\alpha'(G)=\alpha'(G')+1$ and
	\begin{align*}R(G)&=R(G')+\frac12+\frac2{\sqrt 6}+\frac 1{\sqrt{3d(y)}}-\frac 1{\sqrt{d(y)}}\\
		&\geq R(G')+\frac12+\frac3{\sqrt 6}-\frac 1{\sqrt 2} >R(G')+\frac{1}{3}+\frac{1}{\sqrt{3}};\end{align*}
	again, the result follows.
	
	Thus we may assume $t\neq w$. Suppose that $uv$ is contained in a maximum matching. Identify two vertices $u$ and $v$ in $G$ and call this vertex $x$. Add a new vertex $y$ to $G$ and join $y$ to $x$. Call the resultant graph $G^{\prime}$. Note that $\alpha^{\prime}(G)=\alpha^{\prime} (G^{\prime})$. Now, we find that
	\begin{equation*}
		R(G)=R(G^{\prime})+\Big(\frac{1}{\sqrt{2}}-\frac{1}{\sqrt{3}}\Big)\Big(\frac{1}{\sqrt{d(t)}}+\frac{1}{\sqrt{d(w)}}\Big)+\frac{1}{2}-\frac{1}{\sqrt{3}}.
	\end{equation*}
	Thus, by the induction hypothesis, we have
	\begin{equation*}
		R(G) \geq \Big(\frac{1}{3}+\frac{1}{\sqrt{3}}\Big) \alpha^{\prime}(G) + \Big(\frac{1}{2} - \frac{1}{\sqrt{3}}\Big) \frac{2}{\sqrt{3}} + \frac{1}{2} - \frac{1}{\sqrt{3}} > \Big(\frac{1}{3} + \frac{1}{\sqrt{3}}\Big) \alpha^{\prime}(G).
	\end{equation*}
	Now, suppose that $uv$ is not contained in a maximum matching. In this case $G$ contains a maximum matching containing the two edges $tu$ and $vw$. Let $G^{\prime} = G \backslash uv$. Clearly, $\alpha^{\prime}(G)=\alpha^{\prime} (G^{\prime})$. We have
	\begin{align*}
		R(G) & = R(G^{\prime}) + \frac{1}{2} + \Big(\frac{1}{\sqrt{2}}-1\Big)\Big(\frac{1}{\sqrt{d(t)}}+\frac{1}{\sqrt{d(w)}}\Big) \\
		& \geq \Big(\frac{1}{3}+\frac{1}{\sqrt{3}}\Big) \alpha^{\prime}(G)+\frac{1}{2}+\Big(\frac{1}{2}-1\Big)\frac{2}{\sqrt{2}} \\
		& > \Big(\frac{1}{3}+\frac{1}{\sqrt{3}}\Big) \alpha^{\prime}(G).
	\end{align*}
	Hence we may assume that every pendant vertex is adjacent to a vertex of degree $3$ and that there are no two adjacent vertices of degree $2$.
	
	Let $a_i \hspace*{2pt} (i=1,2,3)$ be the number of vertices of degree $i$ in $G$. We have $|G| = \sum_{i=1}^3 a_i$ and $2|E(G)| = a_1 + 2a_2 + 3a_3$. Thus we find that
	\begin{equation*}
		R(G) = \frac{a_1}{\sqrt{3}} + \frac{2a_2}{\sqrt{6}} + \frac{|E(G)| - a_1 - 2a_2}{3}.
	\end{equation*}
	It suffices to prove that
	\begin{equation*}
		\Big(\frac{1}{\sqrt{3}}-\frac{1}{6}-\frac{1}{6}-\frac{1}{2 \sqrt{3}}\Big) a_1 + \Big(\frac{2}{\sqrt{6}}-\frac{1}{3}-\frac{1}{6}-\frac{1}{2 \sqrt{3}}\Big) a_2 +  \Big(\frac{1}{2}-\frac{1}{6}-\frac{1}{2 \sqrt{3}}\Big) a_3 \geq 0.
	\end{equation*}
	Note that if we prove
	\begin{equation*}
		(a_3 - a_1) \Big(\frac{1}{2}-\frac{1}{6}-\frac{1}{2 \sqrt{3}}\Big) \geq 0,
	\end{equation*}
	then we are done. Note that every pendant vertex is adjacent to a vertex of degree $3$ and no pair of pendant vertices have a common neighbor, because otherwise by removing one of them and using induction the assertion holds. Thus $a_3 \geq a_1$, as desired.
	
	Now, assume that $R(G)=\Big(\frac{1}{\sqrt{3}}+\frac{1}{3}\Big) \alpha^{\prime}(G)$. As the proof shows, $a_2=0$ and $a_1=a_3$.  Thus, $a_1=a_3=|G|/2$. If we remove all pendant vertices, then the resultant graph is a $2$-regular graph and therefore a union of cycles.
\end{proof}

\section{Sparse graph classes}\label{sec:sparse}
In this section we turn to sparse graph classes such as planar graphs. Fix a graph $G$; if $A$ is a set of vertices we write $\langle A\rangle$ for the induced subgraph on those vertices.
\begin{thm}\label{high-low}Fix a positive integer $r$. For a graph $G$, let $A=\{v\in V(G):d(v)>r\}$. If the average degree of $\langle A\rangle$ is at most $r$, then $R(G)\geq\frac{\alpha'(G)}{\sqrt{r(r+1)}}$.
\end{thm}
\begin{proof}Write $B=V(G)\setminus A$ and $\ell=|E(\langle B\rangle)|$. Clearly $\alpha'(G)\leq |A|+\ell$. By definition of $B$, the sum of weights of all edges in $B$ is at least $\ell/r$.
	
	For each $u\in A$, write $s(u)=r+1-d_{\langle A\rangle}(u)$. By assumption we have $\sum_{u\in A}d_{\langle A\rangle}(u)\leq r|A|$ and so $\sum_{u\in A}s(u)\geq |A|$. For each $u\in A$ with $s(u)>0$, the sum of weights of edges $ub$ with $b\in B$ is at least
	\[\frac{x}{\sqrt{r(d_{\langle A\rangle}(u)+x)}},\]
	where $x=d(u)-d_{\langle A\rangle}(u)\geq s(u)$ is the number of such edges. This is increasing in $x$, so is at least
	\[\frac{s(u)}{\sqrt{r(d_{\langle A\rangle}(u)+s(u))}}=\frac{s(u)}{\sqrt{r(r+1)}}.\]
	Thus the total weight of all edges between $A$ and $B$ is at least
	\[\sum_{\substack{u\in A\\s(u)>0}}\frac{s(u)}{\sqrt{r(r+1)}}\geq\sum_{u\in A}\frac{s(u)}{\sqrt{r(r+1)}}\geq \frac{|A|}{\sqrt{r(r+1)}}.\]
	Thus \[R(G)\geq \frac{\ell}{r}+\frac{|A|}{\sqrt{r(r+1)}}\geq \frac{\ell+|A|}{\sqrt{r(r+1)}}\geq \frac{\alpha'(G)}{\sqrt{r(r+1)}}.\qedhere\]
\end{proof}
\begin{cor}\label{planar}
	Let $\mathcal G_r$ be the hereditary class of graphs all of whose induced subgraphs have average degree at most $r$. Then 
	\[R(G)\geq\frac{\alpha'(G)}{\sqrt{r(r+1)}}\quad\text{for every }G\in \mathcal G_r.\]
	In particular, since planar graphs are in $\mathcal G_6$, we have \[R(G)\geq\frac{\alpha'(G)}{\sqrt{42}}\quad\text{for every planar graph }G.\]
\end{cor}
\begin{rem}
	It seems plausible that the optimal constant in Corollary \ref{planar} is smaller by approximately a factor of $2$. However, we cannot obtain such a strong improvement merely from sharpening Theorem \ref{high-low}, which is already asymptotically the best possible since the constant there cannot be improved below $1/r$. To see this consider the graph join of $\ell$ independent edges and the empty graph on $r-1$ vertices for $\ell$ large. Note that while this graph satisfies the condition of Theorem \ref{high-low}, it is not in $\mathcal{G}_r$ since its average degree is approximately $2r-1$.
\end{rem}
We observe that our assumption that the class is hereditary is necessary for a result such as Corollary \ref{planar} to hold: no linear bound applies for connected graphs of average degree at most $r>2$. To see this, consider the following construction. Start from the complete bipartite graph $K_{a,b}$, and add one leaf adjacent to each vertex of degree $a$ and $2b$ leaves adjacent to each vertex of degree $b$. This graph $G$ has $(3a+1)b$ edges and $a+2ab+2b$ vertices, so its average degree is less than $3$. We have $\alpha'(G)=a+b$, but
\[R(G)=\frac{b}{\sqrt{a+1}}+\frac{ab}{\sqrt{3b(a+1)}}+\frac{2ab}{\sqrt{3b}}\leq \frac{1}{\sqrt{a}}\alpha'(G),\]
provided $b$ is sufficiently large compared to $a$. By choosing $a$ appropriately, this provides a counterexample to any given linear bound.

\section{General graphs}\label{sec:windmill}
In this section we consider how small the Randi\'c index can be for a general graph of matching number $k$. We give a construction which we conjecture to be qualitatively the best possible. This construction has a nearly-perfect matching, and we prove the conjecture in this case.

We say that a matching in a graph $G$ is \textit{nearly-perfect} if it saturates all but at most one vertex, i.e.\ it has size $\lfloor |G|/2\rfloor$. In this section we give a construction to prove the following.
\begin{prop}\label{prop:windmill}For each $k$ there exists a graph with matching number $k$ and Randi\'c index $(3/2-o(1))(2k)^{2/3}$. Furthermore, there exists such a graph with a nearly-perfect matching.
\end{prop}
Again, Proposition \ref{prop:windmill} (with the condition) is qualitatively the best possible.
\begin{thm}\label{near-perfect}Let $G$ be a graph with a nearly-perfect matching of size $k$. Then 
	\[R(G)\geq 3k^{2/3}/2-\sqrt{k/2}=(3/2-o(1))k^{2/3}.\]
\end{thm}

We conjecture that this assertion holds for all graphs.

\begin{conj}\label{two-thirds}If $\alpha'(G)=k$ then $R(G)\geq ck^{2/3}$ for some absolute constant $c>0$.
\end{conj}
In fact, since the constructions and bounds in this section have negative error terms, it seems plausible that the optimal constant is attained for some small $k$, and so we expect Conjecture \ref{two-thirds} to hold with $c=1$, attained by $K_2$. However, it would also be interesting to determine the behavior of the sequence $(c_k)$ of optimal constants for particular values of $k$, which we expect to have a significantly higher limiting value.
\begin{proof}[Proof of Proposition \ref{prop:windmill}]
We modify the \textit{windmill graph} $\Wm$ (consisting of $k$ triangles sharing a vertex) to define the \textit{generalized windmill} $\Wm[k,r]$, for $k>r\geq 0$ as follows: take the complete bipartite graph $K_{2(k-r),2r+1}$, and add a perfect matching in the even class. Equivalently, start from the windmill graph $\Wm[k-r]$ and replace the central vertex by $2r+1$ copies. Thus $\Wm=\Wm[k,0]$.
\begin{figure}[ht]
	\centering
	\begin{tikzpicture}
		\foreach\x in {0,...,7}
		\foreach\y in {2.5,4.5,6.5}
		\draw (\x,0) -- (\y,-2);
		\draw (2.5,-2) -- (8,0) -- (9,0) -- (2.5,-2);
		\draw (4.5,-2) -- (9,0);
		\draw (6.5,-2) -- (8,0);
		\foreach \x in {0,2,4,6}
		\draw[very thick, red] (\x,0) -- (\x+1,0);
		\draw[very thick, red] (8,0) -- (4.5,-2);
		\draw[very thick, red] (9,0) -- (6.5,-2);
		\foreach \x in {0,...,9}
		\filldraw (\x,0) circle (.05);
		\foreach\y in {2.5,4.5,6.5}
		\filldraw (\y,-2) circle (.05);
	\end{tikzpicture}	
	\caption{The generalized windmill $\Wm[6,1]$, with a maximum matching highlighted.}
\end{figure}
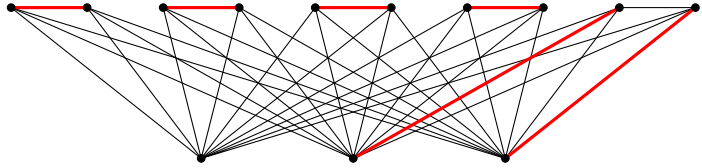

Note that $\alpha'(\Wm[k,r])=k$ provided $r\leq k/2$. We can calculate that \[R(\Wm[k,r])=\frac{k-r}{2r+2}+\frac{(2r+1)(2k-2r)}{\sqrt{(2r+2)(2k-2r)}}.\]
Choosing $r=\lfloor\sqrt[3]{k/4}\rfloor$, we have
\[R(\Wm[k,r])=\frac{k}{2r}+2\sqrt{kr}-O(r)=(3/2-o(1))(2k)^{2/3}.\qedhere\]
\end{proof}
\begin{rem}The smallest graph of this form that violates \eqref{tree-bound} is $\Wm[20,1]$, with $41$ vertices.\end{rem}

\begin{proof}[Proof of Theorem \ref{near-perfect}]
	Fix such a matching $M=\{u_1v_1,\ldots,u_kv_k\}$. For each $i\leq k$, consider the quantity \[r_i=\frac{1}{\sqrt{d(u_i)d(v_i)}}+\frac{1}{2}\sum_{w\in N(u_i)\setminus \{v_i\}}\frac{1}{\sqrt{d(u_i)d(w)}}+\frac{1}{2}\sum_{w\in N(v_i)\setminus \{u_i\}}\frac{1}{\sqrt{d(v_i)d(w)}}.\]
	Note that $\sum_{i=1}^kr_i\leq R(G)$, since every edge of the matching contributes its weight to one term and every other edge contributes half its weight to at most two terms. Also, since there are at most $2k+1$ vertices, we have
	\begin{align*}r_i&\geq \frac{1}{\sqrt{d(u_i)d(v_i)}}+\frac{d(u_i)-1}{2}\cdot\frac{1}{\sqrt{2kd(u_i)}}+\frac{d(v_i)-1}{2}\cdot\frac{1}{\sqrt{2kd(v_i)}}\\
		&\geq\frac{1}{\sqrt{d(u_i)d(v_i)}}+\frac{\sqrt{d(u_i)}}{\sqrt{8k}}+\frac{\sqrt{d(v_i)}}{\sqrt{8k}}-\frac{1}{\sqrt{2k}}\\
		&\geq\frac{3}{2\sqrt[3]{k}}-\frac{1}{\sqrt{2k}},\end{align*}
	by the AM-GM inequality applied to the first three terms.
	Thus
	\[\sum_{i=1}^kr_i\geq \frac{3k^{2/3}}{2}-\frac{\sqrt{k}}{\sqrt{2}},\]
	as required.
\end{proof}
\begin{rem}The bounds of Theorem \ref{near-perfect} and Proposition \ref{prop:windmill} differ by a multiplicative factor of $2^{2/3}\approx 1.59$. We might have expected the leading terms to match, since almost all the individual $r_i$ essentially match the lower bound obtained on each. However, the remaining $r_i$ fail to match that lower bound for two reasons: not only do they have (up to relabelling) $d(v_i)$ large, but also $d(w)$ is small for almost all neighbors $w$ of $v_i$. Thus each of these exceptional matching edges makes a large contribution to the total sum.
\end{rem}

\end{document}